\documentclass [a4paper,12pt]{amsart}
\usepackage{amsthm}
\usepackage{amsmath}
\usepackage{amssymb}
\usepackage[ansinew]{inputenc}
\usepackage{amscd}
\usepackage[ansinew]{inputenc}
\usepackage[mathscr]{euscript}

\newtheorem{thm}{Theorem}[section]
\newtheorem{cor}[thm]{Corollary}
\newtheorem{lem}[thm]{Lemma}
\newtheorem{prop}[thm]{Proposition}

\theoremstyle{definition}
\newtheorem{ex}[thm]{Example}

\theoremstyle{definition}
\newtheorem{defn}[thm]{Definition}

\theoremstyle{definition}
\newtheorem{rem}[thm]{Remark}

\theoremstyle{definition}

\def\Q{\mathbb Q}
\def\C{\mathbb C}
\def\R{\mathbb R}

\def\Z{\mathbb Z}

\def\A{\mathscr A}
\def\dim{\operatorname{dim}}

\def\supp {\mathrm{supp}}

\def\x {\mathbf x}
\def\O{\mathcal O}

\def\B{\mathscr B}

\def\m{\mathit m}

\def\w{\overline w}

\def\LL{\mathscr L}
\def\k{\mathit k}

\def\geq{\geqslant}
\def\leq{\leqslant}

\def\*{\blacksquare}

\subjclass[$2010$ Mathematics Subject Classification]{Primary
32S05; Secondary 13H15}

\begin{document}

\title[The {\L}ojasiewicz exponent of a set of weighted homogeneous
ideals]{The {\L}ojasiewicz exponent of a set of weighted\\\vskip4pt
homogeneous ideals}
\author{C. Bivi\`a-Ausina}
\address{
Institut Universitari de Matem\`atica Pura i Aplicada,
Universitat Polit\`ecnica de Val\`encia,
Cam\'i de Vera, s/n,
46022 Val\`encia,
Spain}
\email{carbivia@mat.upv.es}

\author{S. Encinas}
\address{Departamento de Matem\'atica Aplicada,
Universidad de Valladolid,
Avda. Salamanca s/n,
47014 Valladolid,
Spain}
\email{sencinas@maf.uva.es}


\thanks{The first author was partially supported by DGICYT
Grant MTM2006--06027. The second author was partially supported by
DGICYT Grant MTM2006--10548.}

\begin{abstract}
We give an expression for the {\L}ojasiewicz exponent of a set of
ideals
which are pieces of a weighted homogeneous filtration. We also study
the application of this formula
to the computation of the {\L}ojasiewicz exponent of the gradient of a
semi-weighted homogeneous function $(\C^n,0)\to (\C,0)$ with an isolated
singularity at the origin.
\end{abstract}

\maketitle

\section{Introduction}

Let $R$ be a Noetherian ring and let $I$ be an ideal of $R$.  Let
$\nu_I$ be the order function of $R$ with respect to $I$, that is,
$\nu_I(h)=\sup\{r: h\in I^r\}$, for all $h\in R$, $h\neq 0$ and
$\nu(0)=\infty$.  Let us consider the function $\overline{\nu}_I:R\to
\R_{ \geq 0}\cup \{\infty\}$ defined by
$\overline{\nu}_I(h)=\lim_{s\to \infty}\frac{\nu_I(h^s)}{s}$, for all
$h\in R$.  It was proven by Samuel \cite{Samuel1952} and Rees
\cite{Rees1956}
that this limit exists and Nagata proved in \cite{Nagata1957} that, when
finite, the number $\nu_I(h)$ is a rational number.  The function
$\overline\nu$ is called the {\it asymptotic Samuel function of $I$}.
If $J$ is another ideal of $R$, then the number $\overline \nu_I(J)$
is defined analogously and if $h_1,\dots, h_r$ is a generating system
of $J$ then $\overline \nu_I(J)=\min\{\overline \nu_I(h_1),\dots,
\overline \nu_I(h_r)\}$.  Let us denote by $\overline I$ the integral
closure of $I$.  As a consequence of the theorem of existence of the
Rees valuations of an ideal (see for instance \cite[p.\ 192]{HunekeSwanson2006}), it
is known that, if $J$ is another ideal and $p,q\in\Z_{\geq 1}$, then
$J^q\subseteq \overline{I^p}$ if and only if $\overline \nu_I(J)\geq
\frac{p}{q}$.

Let $\O_n$ denote the ring of analytic function germs $f:(\C^n,0)\to
\C$ and let $m_n$ denote its maximal ideal, that will be also denoted
by $m$ if no confusion arises.  Let $I$ be an ideal of $\O_n$ of
finite colength.  Then Lejeune and Teissier proved in
\cite[p.\ 832]{LejeuneTeissier1974} that $\frac{1}{\overline \nu_I(m)}$
is equal to the {\L}ojasiewicz exponent of $I$ (in fact, this result was proven
in a more general context, that is, for ideals in a structural ring $\O_X$,
where $X$ is a reduced complex analytic space).  If $g_1,\dots, g_r$
is a generating system of $I$, then the {\it {\L}ojasiewicz exponent of
$I$} is defined as the infimum of those $\alpha>0$ such that there
exists a constant $C>0$ and an open neighbourhood $U$ of $0\in\C^n$
such that
$$
\Vert x\Vert^\alpha\leq C \sup_i \vert g_i(x)\vert,
$$
for all $x\in U$.  Let us denote this number by $\LL_0(I)$ and let
$e(I)$ denote the Samuel multiplicity of $I$.  Therefore we have that
$\LL_0(I)=\inf\{\frac pq\in\Q_+: m^p\subseteq \overline{I^q}\}$ and
hence, by the Rees' multiplicity theorem (see \cite[p.\ 222]{HunekeSwanson2006}) we
have that $\LL_0(I)=\inf\{\frac pq\in\Q_+: e(I^q)=e(I^q+m^p)\}$.  This
expression of $\LL_0(I)$ is one of the motivations that lead the first
author to introduce the notion of {\L}ojasiewicz exponent of a set of
ideals in \cite{Bivia2009}.  This notion is based on the Rees' mixed
multiplicity of a set of ideals (Definition \ref{lasigma}).

{\L}ojasiewicz exponents have important applications in singularity
theory.  Here we recall one of them.  Let $f:(\C^n,0)\to (\C,0)$ be
the germ of a complex analytic function with an isolated singularity
at the origin and let $J(f)=\langle\frac{\partial f}{\partial
x_1},\dots, \frac{\partial f}{\partial x_n}\rangle$ be the Jacobian
ideal of $f$.  Let us denote the number $\LL_0(J(f))$ by $\LL_0(f)$.
The {\it degree of $C^0$-determinacy of $f$}, denoted by $s_0(f)$, is
defined as the smallest integer $r$ such that $f$ is topologically
equivalent to $f+g$, for all $g$ such that $\nu_{m_n}(g)\geq r+1$.
Teissier proved in \cite[p.\ 280]{Teissier1977} that
$s_0(f)=[\LL_0(f)]+1$, where $[a]$ stands for the integer part of a
given $a\in\R$.  Despite the fact that this equality connects
$\LL_0(f)$ with a fundamental topological aspect of $f$, the problem
of determining whether the {\L}ojasiewicz exponent $\LL_0(f)$ is a
topological invariant of $f$ is still an open problem.

The effective computation of $\LL_0(I)$ has proven to be a challenging
problem in algebraic geometry that, by virtue of the results of
Lejeune and Teissier is directly related with the computation of the
integral closure of an ideal.  In \cite{BiviaEncinas2009} the authors relate the
problem of computing $\LL_0(I)$ with the algorithms of resolution of
singularities.  The approach that we give in this paper is based on
techniques of commutative algebra.

We recall that, if $w=(w_1,\dots, w_n)\in \Z_{\geq 1}^n$, then a
polynomial $f\in \C[x_1,\dots, x_n]$ is called weighted homogeneous of
degree $d$ with respect to $w$ when $f$ is written as a sum of
monomials $x_1^{k_1}\cdots x_n^{k_n}$ such that
$w_1x_1+\cdots+w_nx_n=d$.  This paper is motivated by the main result
of Krasi\'nski, Oleksik and P{\l}oski in \cite{KrasinskiOleksikPloski2009},
which says that if
$f:\C^3\to \C$ is a weighted homogeneous polynomial of degree $d$ with
respect to $(w_1,w_2,w_3)$ with an isolated singularity at the origin,
then $\LL_0(f)$ is given by the expression
$$
\LL_0(f)=\frac{d-\min\{w_1,w_2,w_3\}}{\min\{w_1,w_2,w_3\}},
$$
provided that $d\geq 2w_i$, for all $i=1,2,3$.  That is, $\LL_0(f)$
depends only on the weights $w_i$ and the degree $d$ in this case.
Therefore it is concluded that $\LL_0(f)$ is a topological invariant
of $f$, by virtue of the results of Saeki \cite{Saeki1988} and
Yau \cite{Yau1988}.
In view of the above equality it is reasonable to
conjecture that the analogous result holds in general, that is, if
$f:(\C^n,0)\to (\C,0)$ is a weighted homogeneous polynomial, or even a
semi-weighted homogeneous function (see Definition \ref{semiqh}), with
respect to $(w_1,\dots, w_n)$ of degree $d$ with an isolated
singularity at the origin, and if $d\geq 2w_i$, for all $i=1,\dots,
n$, then
\begin{equation}\label{conjectura}
\LL_0(f)=\frac{d-\min\{w_1,\dots,w_n\}}{\min\{w_1,\dots,w_n\}}.
\end{equation}
We point out that inequality $(\leq)$ always holds in
(\ref{conjectura}) for semi-weighted homogeneous functions (see
Corollary \ref{CorMain}).

In this paper we obtain the equality (\ref{conjectura}) for a
semi-weighted homogeneous germs $f:(\C^n,0)\to (\C,0)$ under a
restriction expressed in terms of the supports of the component
functions of $f$ (see Corollary \ref{CorMain}).  This result arises as
a consequence of a more general result involving the {\L}ojasiewicz
exponent of a set of ideals coming from a weighted homogeneous
filtration (see Theorem \ref{main}).  Our approach to {\L}ojasiewicz
exponents is purely algebraic and comes from the techniques developed
in \cite{Bivia2008} and \cite{Bivia2009}.  This new point of view of the
subject has led us to detect a broad class of semi-weighted
homogeneous functions where relation (\ref{conjectura}) holds.

\section{The Rees' mixed multiplicity of a set of ideals}

Let $(R,m)$ be a Noetherian local ring and let $I$ be an ideal of
$R$. We denote by $e(I)$ the
Samuel multiplicity of $I$. Let $\dim R=n$ and let us fix a set of
$n$ ideals
$I_1,\dots, I_n$ of $R$ of finite colength. Then we denote
by $e(I_1,\dots, I_n)$ the mixed multiplicity of $I_1,\dots, I_n$, as
defined by Teissier and Risler in \cite{Teissier1973} (we refer to
\cite[\S 17]{HunekeSwanson2006} and \cite{Swanson2007} for fundamental results about mixed
multiplicities of ideals). We recall that, if the ideals $I_1,\dots,
I_n$ are equal to a
given ideal, say $I$, then $e(I_1,\dots, I_n)=e(I)$.

Let us suppose that the residue field $k=R/m$ is infinite. Let
$a_{i1},\dots, a_{is_i}$ be
a generating system of $I_i$, where $s_i\geqslant 1$, for
$i=1,\dots, n$. Let $s=s_1+\cdots +s_n$. We say that a property
holds for {\it sufficiently general} elements of $I_1\oplus \cdots
\oplus I_n$ if there exists a non-empty Zariski-open set $U$ in
$k^s$ such that the said property holds for all elements
$(g_1,\dots, g_n)\in I_1\oplus \cdots \oplus I_n$ such that
$g_i=\sum_{j}u_{ij}a_{ij}$, $i=1,\dots, n$ and the image of $(u_{11}, \dots,
u_{1s_1},\dots, u_{n1}, \dots, u_{ns_n})$ in $k^{s}$ lies in $U$.

By virtue of a result of Rees (see \cite{Rees1984} or
\cite[p.\ 335]{HunekeSwanson2006}), if the ideals $I_1,\dots, I_n$
have finite colength, then
the mixed multiplicity of $I_1,\dots, I_n$ is obtained as
$e(I_1,\dots, I_n)=e(g_1,\dots, g_n)$, for a
sufficiently general element $(g_1,\dots, g_n)\in
I_1\oplus\cdots\oplus I_n$.

Let us denote by $\O_n$ the ring of analytic function germs
$(\C^n,0)\to \C$.  Let $g:(\C^n,0)\to (\C^n,0)$ be a
complex analytic map germ such that $g^{-1}(0)=\{0\}$ and let $g_1,\dots,
g_n$ denote the component functions of $g$. We recall that $e(I)=\dim_\C
\O_n/I$, where $I$ is the ideal of $\O_n$ generated by $g_1,\dots,
g_n$. It turns that this number is equal to the geometric multiplicity
of $g$ (see \cite[p.\ 258]{Lojasiewicz1991} or \cite{Palamodov1967}).

Now we show the definition of a number associated to a family of
ideals that
generalizes the notion of mixed multiplicity. This number is
fundamental in the results of this paper.

We denote by $\Z_+$ the set of non-negative integers. Let $a\in \Z$,
we denote by $\Z_{\geq a}$ the set of integers $z\geq a$.

\begin{defn}\cite{Bivia2008}\label{lasigma}
Let $(R,m)$ be a Noetherian local ring of dimension $n$. Let
$I_1,\dots,
I_n$ be ideals of $R$. Then we define the {\it Rees' mixed
multiplicity of $I_1,\dots, I_n$ as}
\begin{equation}\label{sigma}
\sigma(I_1,\dots, I_n)=\max_{r\in\Z_+}\,e(I_1+\m^r,\dots,
I_n+\m^r),
\end{equation}
when the number on the right hand side is finite. If the set of
integers
$\{e(I_1+\m^r,\dots, I_n+\m^r): r\in\Z_+\}$ is non-bounded then we
set $\sigma(I_1,\dots, I_n)=\infty$.
\end{defn}

We remark that if $I_i$ is an ideal of finite colength, for all
$i=1,\dots, n$, then
$\sigma(I_1,\dots, I_n)=e(I_1,\dots, I_n)$. The next proposition
characterizes the finiteness of
$\sigma(I_1,\dots, I_n)$.

\begin{prop}\cite[p.\,393]{Bivia2008}\label{sigmaexists}
Let $I_1,\dots, I_n$ be ideals of a Noetherian
local ring $(R,m)$ such that the residue field $\k=R/m$ is
infinite. Then $\sigma(I_1,\dots, I_n)<\infty$ if and only if
there exist elements $g_i\in I_i$, for $i=1,\dots, n$, such that
$\langle g_1,\dots, g_n\rangle$ has finite colength. In this case,
we have that $\sigma(I_1,\dots, I_n)=e(g_1,\dots, g_n)$ for
sufficiently general elements $(g_1,\dots, g_n)\in I_1\oplus
\cdots \oplus I_n$.
\end{prop}

The following result will be useful in subsequent sections.

\begin{lem}\label{reverseincl}\cite[p. 392]{Bivia2009}
Let $(R,m)$ be a Noetherian local ring of dimension $n\geq 1$. Let
$J_1,\dots, J_n$ be ideals of $R$ such that
$\sigma(J_1,\dots, J_n)<\infty$. Let $I_1,\dots, I_n$ be ideals of
$R$ such that $J_i\subseteq I_i$, for all $i=1,\dots, n$. Then
$\sigma(I_1,\dots, I_n)<\infty$ and
$$
\sigma(J_1,\dots, J_n)\geqslant \sigma(I_1,\dots, I_n).
$$
\end{lem}

\begin{rem}\label{sigmafinita}
It is worth to point out that, if $I_1,\dots,I_n$ is a set of ideals
of $R$ such that
$\sigma(I_1,\dots, I_n)<\infty$, then $I_1+\cdots+I_n$ is an ideal of
finite colength. Obviously the converse is not true.
\end{rem}

Now we recall some basic definitions.  Let us fix a coordinate system
$x_1,\dots,x_n$ in $\C^n$.  If $k=(k_1,\dots, k_n)\in\Z^n_+$, the we
will denote the monomial $x_1^{k_1}\cdots x_n^{k_n}$ by $x^k$.
If $h\in\O_n$ and $h=\sum_ka_kx^k$ denotes the Taylor expansion of $h$
around the origin, then the {\it support of $h$} is the set
$\supp(h)=\{k\in\Z_+:a_k\neq 0\}$.  If $h\neq 0$, the {\it Newton
polyhedron of $h$}, denoted by $\Gamma_+(h)$, is the convex hull of
the set $\{k+v: k\in\supp(h), v\in\R^n_+\}$.  If $h=0$, then we set
$\Gamma_+(h)=\emptyset$.  If $I$ is an ideal of $\O_n$ and
$g_1,\dots,g_s$ is a generating system of $I$, then we define the {\it
Newton polyhedron of $I$} as the convex hull of $\Gamma_+(g_1)\cup
\cdots \cup\Gamma_+(g_r)$.  It is easy to check that the definition of
$\Gamma_+(I)$ does not depend on the chosen generating system of $I$.
We say that $I$ is a {\it monomial ideal} of $\O_n$ when $I$ admits a
generating system formed by monomials.

\begin{defn}
Let $I_1,\dots, I_n$ be monomial ideals of $\O_n$ such that
$\sigma(I_1,\dots, I_n)<\infty$.  Then we denote by $\mathcal
S(I_1,\dots, I_n)$ the family of those maps $g=(g_1,\dots,
g_n):(\C^n,0)\to (\C^n,0)$ such that $g^{-1}(0)=\{0\}$, $g_i\in I_i$,
for all $i=1,\dots, n$, and $\sigma(I_1,\dots, I_n)=e(g_1,\dots,
g_n)$, where $e(g_1,\dots, g_n)$ stands for the multiplicity of the
ideal of $\O_n$ generated by $g_1,\dots, g_n$.  The elements of
$\mathcal S(I_1,\dots, I_n)$ are characterized in
\cite[Theorem 3.10]{Bivia2008}.

We denote by $\mathcal S_0(I_1,\dots, I_n)$ the set formed by the maps
$g=(g_1,\dots, g_n)\in \mathcal S(I_1,\dots, I_n)$ such that
$\Gamma_+(g_i)=\Gamma_+(I_i)$, for all $i=1,\dots, n$.
\end{defn}

\section{The {\L}ojasiewicz exponent of a set of ideals}

If $g:(\C^n,0)\to (\C^n,0)$ is an analytic map germ such that
$g^{-1}(0)=\{0\}$, then we denote by $\LL_0(g)$ the {\L}ojasiewicz
exponent of the ideal of $\O_n$ generated by the component functions
of $g$.

Let $I_1,\dots, I_n$ be ideals of a local ring $(R,m)$ such that
$\sigma(I_1,\dots, I_n)<\infty$.  Then we define
\begin{equation}\label{laerre}
r(I_1,\dots, I_n)=\min\big\{r\in\Z_+: \sigma(I_1,\dots,
I_n)=e(I_1+m^r,\dots, I_n+m^r)\big\}.
\end{equation}

\begin{thm}\textnormal{\cite[p.\ 398]{Bivia2009}}\label{base}
Let $I_1,\dots, I_n$ be monomial ideals of $\O_n$ such that
$\sigma(I_1,\dots,
I_n)$ is finite. If $g\in\mathcal S_0(I_1,\dots, I_n)$, then $\mathscr
L_0(g)$ depends only on $I_1,\dots, I_n$ and it is given by:
\begin{equation}\label{intrinsec}
\mathscr L_0(g)=\min_{s\geqslant 1}\frac{r(I_1^s,\dots,
I_n^s)}{s}.
\end{equation}
\end{thm}

By the proof of the above theorem it is concluded that the infimum of
the sequence $\{\frac{r(I_1^s,\dots, I_n^s)}{s}\}_{s\geq 1}$ is
actually a minimum.  Theorem \ref{base} motivates the following
definition.

\begin{defn}
Let $(R,m)$ be a Noetherian local ring of dimension $n$. Let
$I_1,\dots,
I_n$ be ideals of $R$. Let us suppose that $\sigma(I_1,\dots,
I_n)<\infty$. We define the {\it {\L}ojasiewicz exponent of
$I_1,\dots, I_n$} as
$$
\mathscr L_0(I_1,\dots, I_n)=\inf_{s\geq 1}\frac{r(I_1^s,\dots,
I_n^s)}{s}.
$$
\end{defn}

As we will see in Lemma \ref{rpowers}, we have that
$r(I_1^s,\dots, I_n^s)\leq s r(I_1,\dots, I_n)$, for all $s\in \Z_{\geq 1}$.
Therefore $\LL_0(I_1,\dots, I_n)\leq r(I_1,\dots, I_n)$.

We can extend Definition \ref{lasigma} by replacing the maximal ideal
$m$ by an arbitrary ideal of finite colength,
but the resulting number is the same. That is, under the hypothesis
of Definition \ref{lasigma}, let us denote
by $J$ an ideal of $R$ of finite colength and let us suppose that
$\sigma(I_1,\dots, I_n)<\infty$. Then we define
$$
\sigma_J(I_1,\dots, I_n)=\max_{r\in\Z_+}\,e(I_1+J^r,\dots, I_n+J^r).
$$
An easy computation reveals that $\sigma_J(I_1,\dots,
I_n)=\sigma(I_1,\dots, I_n)$. We also define
\begin{equation}\label{laerresubJ}
r_J(I_1,\dots, I_n)=\min\big\{r\in\Z_+: \sigma(I_1,\dots,
I_n)=e(I_1+J^r,\dots, I_n+J^r)\big\}.
\end{equation}

Let $I$ be an ideal of $R$ of finite colength. Then we denote by $r_J(I)$ the number
$r_J(I,\dots, I)$, where $I$ is repeated $n$ times.  We deduce from
the Rees' multiplicity theorem that, if $R$ is quasi-unmixed, then $r_J(I)=\min\{r\geq 1: J^r\subseteq
\overline I\}$.

\begin{lem}\label{rpowers}
Let $(R,m)$ be a Noetherian local ring of dimension $n$. Let
$I_1,\dots,
I_n$ be ideals of $R$ such that $\sigma(I_1,\dots, I_n)<\infty$ and
let $J$ be an $m$-primary ideal. Then
\begin{align*}
r_J(I_1^s,\dots, I_n^s)&\leq sr_J(I_1,\dots, I_n)\\
r_{J^s}(I_1,\dots, I_n)&\geq \frac{1}{s}r_J(I_1,\dots, I_n),
\end{align*}
for all integer $s\geq 1$.
\end{lem}

\begin{proof}
For the first inequality, set $r=r_{J}(I_{1},\ldots,I_{n})$. So
that $\sigma(I_{1},\ldots,I_{n})= e(I_{1}+J^{r},\ldots,I_{n}+J^{r})$.
It is enough to prove that
$\sigma(I_{1}^{s},\ldots,I_{n}^{s})=
e(I_{1}^{s}+J^{rs},\ldots,I_{n}^{s}+J^{rs})$:
\begin{align*}
    e(I_{1}^{s}+J^{rs},\ldots,I_{n}^{s}+J^{rs}) &=
    e(\overline{I_{1}^{s}+J^{rs}},\ldots,\overline{I_{n}^{s}+J^{rs}})
     =
     e(\overline{(I_{1}+J^{r})^{s}},\ldots,\overline{(I_{n}+J^{r})^{s}}) \\
     &= e((I_{1}+J^{r})^{s},\ldots,(I_{n}+J^{r})^{s})
     = s^{n}e(I_{1}+J^{r},\ldots,I_{n}+J^{r}) \\
     &= s^{n}\sigma(I_{1},\ldots,I_{n})
     =\sigma(I_{1}^{s},\ldots,I_{n}^{s}),
\end{align*}
where last equality comes from \cite[Lemma 2.6]{Bivia2009}.

The second inequality comes directly from the definition of
$r_{J^{s}}(I_{1},\ldots,I_{n})$.
\end{proof}

It is easy to
find examples of ideals $I$ and $J$
such that $r_J(I_1,\dots, I_n)\neq r(I_1,\dots, I_n)$ in general.
This fact motivates the following definition.

\begin{defn}
Let $(R,m)$ be a Noetherian local ring of dimension $n$. Let
$I_1,\dots,
I_n$ be ideals of $R$ such that $\sigma(I_1,\dots, I_n)<\infty$. Let
$J$ be an $m$-primary ideal of $R$.
We define the {\it {\L}ojasiewicz exponent of $I_1,\dots, I_n$ with
respect to $J$}, denoted by $\LL_J(I_1,\dots, I_n)$, as
\begin{equation}\label{LJI}
\LL_J(I_1,\dots, I_n)=\inf_{s\geq 1}\frac{r_J(I_1^s,\dots, I_n^s)}{s}.
\end{equation}
If $I$ is an $m$-primary ideal of $R$, then we denote by $\LL_J(I)$
the number $\LL_J(I,\dots, I)$, where $I$ is repeated $n$ times.
\end{defn}

\begin{rem} \label{RemLimitInf}
Under the conditions of the previous definition, we observe that
$\LL_{J}(I_{1},\ldots,I_{n})$ can be seen as an inferior limit:
\begin{equation*}
    \LL_{J}(I_{1},\ldots,I_{n})=
    \liminf_{s\to\infty}\frac{r_J(I_1^s,\dots, I_n^s)}{s}.
\end{equation*}
Set $\ell=\LL_{J}(I_{1},\ldots,I_{n})$.  In order to prove the
equality above, it is enough to see that for all $\epsilon>0$ and all
$p\in \Z_+$, there exists an integer $m\geq p$ such that
\begin{equation*}
    \frac{r_J(I_1^m,\dots, I_n^m)}{m}\leq
    \ell+\epsilon.
\end{equation*}
Let us fix an $\epsilon>0$ and an integer $p\in\Z_+$.  By
definition, there exists $q\in\Z_+$ such that $$
\dfrac{r_{J}(I_{1}^{q},\ldots,I_{n}^{q})}{q}\leq
\ell+\epsilon. $$
Let $s\in \Z_+$ such that $sq\geq p$.
Then, from Lemma \ref{rpowers} we obtain that
\begin{equation*}
    \frac{r_J(I_1^{sq},\dots, I_n^{sq})}{sq}\leq
    \frac{r_J(I_1^{q},\dots, I_n^{q})}{q}\leq \ell+\epsilon.
\end{equation*}
\end{rem}

A straightforward reproduction of the argument in the proof of Theorem
\ref{base} consisting on replacing the powers of the maximal ideal by
the powers of given ideal of finite colength leads to the following
result, which is analogous to Theorem \ref{base}.

\begin{thm}\label{base2}
Let $I_1,\dots, I_n$ be monomial ideals of $\O_n$ such that
$\sigma(I_1,\dots, I_n)$ is finite and let $J$ be a monomial ideal of
$\O_n$ of finite colength.  Then the sequence
$\{\frac{r_J(I_1^s,\dots, I_n^s)}{s}\}_{s\geq 1}$ attains a minimum
and if $g\in\mathcal S_0(I_1,\dots, I_n)$ then
\begin{equation}\label{intrinsec2}
\mathscr L_J(g)=\LL_J(I_1,\dots, I_n)=\min_{s\geqslant 1}\frac{r_J(I_1^s,\dots, I_n^s)}{s}.
\end{equation}
\end{thm}

\begin{lem}\label{Lpowers}
Under the hypothesis of Lemma \textnormal{\ref{rpowers}} we have
\begin{align*}
\LL_J(I_1^s,\dots, I_n^s)&= s\LL_J(I_1,\dots, I_n)\\
\LL_{J^s}(I_1,\dots, I_n)&= \frac{1}{s}\LL_J(I_1,\dots, I_n),
\end{align*}
for all $s\in\Z_{\geq 1}$.
\end{lem}

\begin{proof}
For the first equality
\begin{equation*}
    \LL_{J}(I_{1}^{s},\ldots,I_{n}^{s})=
    \inf_{p\geq 1}\frac{r_J(I_1^{sp},\dots, I_n^{sp})}{p}=
    s\inf_{p\geq 1}\frac{r_J(I_1^{sp},\dots, I_n^{sp})}{sp}\geq
    s\LL_{J}(I_{1},\ldots,I_{n}).
\end{equation*}
On the other hand, by Lemma \ref{rpowers} we obtain
\begin{equation*}
    \inf_{p\geq 1}\frac{r_J(I_1^{sp},\dots, I_n^{sp})}{p}\leq
    s\inf_{p\geq 1}\frac{r_J(I_1^{p},\dots, I_n^{p})}{p}=
    s\LL_{J}(I_{1},\ldots,I_{n}).
\end{equation*}
Let us see the second equality. Applying Lemma \ref{rpowers} we have
\begin{equation*}
    \LL_{J^{s}}(I_{1},\ldots,I_{n})=
    \inf_{p\geq 1}\frac{r_{J^{s}}(I_1^{p},\dots, I_n^{p})}{p}\geq
    \frac{1}{s}\inf_{p\geq 1}\frac{r_{J}(I_1^{p},\dots, I_n^{p})}{p}=
    \frac{1}{s}\LL_{J}(I_{1},\ldots,I_{n}).
\end{equation*}
Let us denote the number $r_{J^{s}}(I_{1}^{p},\ldots,I_{n}^{p})$ by
$r_p$, for all $p\geq 1$.  Then
\begin{equation*}
    \sigma(I_{1}^{p},\ldots,I_{n}^{p})>
    e(I_{1}^{p}+J^{s(r_{p}-1)},\ldots,I_{n}^{p}+J^{s(r_{p}-1)}).
\end{equation*}
In particular
$$
r_{J}(I_{1}^{p},\ldots,I_{n}^{p})>s(r_{p}-1),
$$
for all $p\geq 1$.  Dividing the previous inequality by $p$ and
taking $\liminf_{p\to\infty}$ we obtain by Remark \ref{RemLimitInf},
that
\begin{equation*}
    \LL_{J}(I_{1},\ldots,I_{n})=
    \liminf_{p\to\infty}\frac{r_{J}(I_{1}^{p},\ldots,I_{n}^{p})}{p}\geq
    s\liminf_{p\to\infty}\left(\frac{r_{p}-1}{p}\right)=
    s\LL_{J^{s}}(I_{1},\ldots,I_{n}).
\end{equation*}
\end{proof}

\begin{lem}\label{transit}
Let $(R,m)$ be a quasi-unmixed Noetherian local ring of dimension
$n$. Let $I_1,\dots,
I_n$ be ideals of $R$ such that $\sigma(I_1,\dots, I_n)<\infty$. If
$J_1, J_2$ are $m$-primary ideals of $R$ then
$$
\LL_{J_1}(I_1,\dots, I_n) \leq \LL_{J_1}(J_2)\LL_{J_2}(I_1,\dots,
I_n).
$$
\end{lem}

\begin{proof}
By (\ref{laerresubJ}) we have that
$$
r_{J_1}(J_2)=\min\big\{r\geq 1: e(J_2)=e(J_2+J_1^r)\big\}.
$$
Given an integer $r\geq 1$, the condition $e(J_2)=e(J_2+J_1^r)$ is
equivalent to saying that
$J_1^r\subseteq \overline{J_2}$, by the Rees' multiplicity theorem
(see \cite[p.\ 222]{HunekeSwanson2006}).
Therefore, an elementary computation shows that
\begin{equation}\label{J2}
r_{J_1}(I_1,\dots, I_n) \leq  r_{J_1}(J_2)r_{J_2}(I_1,\dots, I_n).
\end{equation}
By the generality of the previous inequality, we have
\begin{equation}\label{J2p}
r_{J_1}(I_1^s,\dots, I_n^s) \leq
r_{J_1}(J_2^p)r_{J_2^p}(I_1^s,\dots, I_n^s),
\end{equation}
for all integers $p,s\geq 1$. The inequality (\ref{J2p}) shows that
\begin{align*}
\LL_{J_1}(I_1,\dots, I_n)&=
\inf_{s\geq 1}\frac{r_{J_1}(I_1^s,\dots,I_n^s)}{s}\leq
\inf_{s\geq 1}\frac{r_{J_1}(J_2^p)r_{J_2^p}(I_1^s,\dots, I_n^s)}{s}=\\
&=r_{J_1}(J_2^p)\LL_{J_2^p}(I_1,\dots, I_n)=
r_{J_1}(J_2^p)\frac{1}{p}\LL_{J_2}(I_1,\dots, I_n),
\end{align*}
for all integer $p\geq 1$, where the last equality comes from
Lemma \ref{Lpowers}. Then
$$
\LL_{J_1}(I_1,\dots, I_n)\leq \bigg(\inf_{p\geq 1}
\frac{r_{J_1}(J_2^p)}{p}\bigg)\LL_{J_2}(I_1,\dots,
I_n)=\LL_{J_1}(J_2)\LL_{J_2}(I_1,\dots, I_n).
$$
\end{proof}

\begin{prop}\label{uppers}\cite{Bivia2009} Let $(R,m)$ be a Noetherian
local ring of dimension $n$.
For each $i=1,\dots, n$ let us consider ideals $I_i$ and $J_i$
such that $I_i\subseteq J_i$. Let suppose that $\sigma(I_1,\dots,
I_n)<\infty$ and that $\sigma(I_1,\dots, I_n)=\sigma(J_1,\dots,
J_n)$. Then
\begin{equation}\label{monot}
\LL_0(I_1,\dots, I_n)\leqslant \LL_0(J_1,\dots, J_n).
\end{equation}
\end{prop}

Let us denote the canonical basis in $\R^n$ by $e_1,\dots, e_n$.

\begin{prop}\label{intersecteixos}\cite{Bivia2005}
Let $J$ be an ideal of finite colength of $\O_n$ and set 
$r_i=\min\{r:re_i\in\Gamma_+(J)\}$, for all $i=1,\dots, n$.
Then
$$
\max\{r_1,\dots, r_n\}\leq \LL_0(J)
$$
and equality holds if $\overline J$ is a monomial ideal.
\end{prop}

\section{Weighted homogeneous filtrations}

Let us fix a vector $w=(w_1,\dots, w_n)\in \Z^n_{\geq 1}$.  We will
usually refer to $w$ as the {\it vector of weights}.  Let $h\in\O_n$,
$h\neq 0$, the {\it degree of $h$ with respect to $w$}, or {\it
$w$-degree} of $h$, is defined as
$$
d_w(h)=\min\{\langle k,w\rangle: k\in\supp(h)\},
$$
where $\langle \,,\rangle$ stands for the usual scalar product. In
particular, if $x_1,\dots, x_n$ denotes a system of coordinates in
$\C^n$ and $x_1^{k_1}\dots x_n^{k_n}$ is a monomial in $\O_n$, then
$d_w(x_1^{k_1}\dots x_n^{k_n})=w_1k_1+\cdots +w_nk_n$. By convention,
we set $d_w(0)=+\infty$. If $h\in \O_n$ and $h=\sum_k a_kx^k$ is the
Taylor expansion of $h$ around the origin, then we define the {\it
principal part of $h$ with respect to $w$} as the polynomial given by
the sum of those terms $a_kx^k$ such that $\langle
k,w\rangle=d_w(h)$. We denote this polynomial by $p_w(h)$.

\begin{defn}\label{semiqh}
We say that a function $h\in\O_n$ is {\it weighted homogeneous of
degree $d$ with respect to $w$} if $\langle k,w\rangle=d$, for all
$k\in \supp(h)$.
The function $h$ is said to be {\it semi-weighted homogeneous of
degree $d$ with respect to $w$} when $p_w(h)$ has an isolated
singularity at the origin.
\end{defn}

It is well-known that, if $h$ is a
semi-weighted homogeneous function, then $h$ has an isolated
singularity at the origin and that $h$ and $p_w(h)$ have the same
Milnor number (see for instance \cite[\S 12]{ArnoldGuseinVarchenko1985}).
Let $g=(g_1,\dots, g_n):(\C^n,0)\to (\C^n,0)$ be an analytic map
germ, let us denote the map $(p_w(g_1),\dots, p_w(g_n))$ by $p_w(g)$.
The map $g$ is said to be {\it semi-weighted homogeneous with respect to
$w$} when $(p_w(g))^{-1}(0)=\{0\}$.

If $I$ is an ideal of $\O_n$, then we define the {\it degree of $I$
with respect to $w$} as
$$
d_w(I)=\min\{d_w(h):h \in I\}.
$$
If $g_1,\dots, g_r$ constitutes a generating system of $I$, then it
is straightforward to see that $d_w(I)=\min\{d_w(g_1),\dots,
d_w(g_r)\}$.

Let $r\in\Z_+$, then we denote by $\B_r$ the set of all $h\in\O_n$
such that $d_w(h)\geq r$ (therefore $0\in\B_r$). We observe that
\begin{enumerate}
\item[(a)] $\B_r$ is an integrally closed monomial ideal of finite
colength, for all $r\geq 1$;
\item[(b)] $\B_r\B_s\subseteq \B_{r+s}$, $r,s\geq 1$;
\item[(c)] $\B_0=\O_n$.
\end{enumerate}
The family of ideals $\{\B_r\}_{r\geq 1}$ is called the {\it weighted
homogeneous filtration induced by $w$}.
We denote by $\A_r$ the ideal of $\O_n$ generated by the monomials
$x^k$ such that $d_w(x^k)=r$. If there is not any monomial $x^k$ such
that $d_w(x^k)=r$ then we set $\A_r=0$. Given an integer $r\geq 1$,
we observe that $\A_r\subseteq \B_r$ and that $\overline {\A_r}\neq
\B_r$ in general. Moreover it follows easily that $\overline
{\A_r}=\B_r$ if and only if $\A_r$ is an ideal of finite colength of
$\O_n$.

If $r_1,\dots, r_n\in\Z_{\geq 1}$, then it is not true in general
that $\sigma(\A_{r_1},\dots, \A_{r_n})<\infty$, even if $\A_{r_i}\neq
0$, for all $i=1,\dots, n$. However $\sigma(\B_{r_1},\dots,
\B_{r_n})<\infty$, since $\B_{r_i}$ has finite colength, for all
$i=1,\dots, n$. For instance, let us consider the vector $w=(3,1)$. Then we have
$$
\A_4=\langle xy, y^4\rangle, \hspace{1cm} \A_5=\langle xy^2,
y^5\rangle.
$$
We observe that the ideal $\A_4+\A_5$ has not finite colength, therefore
$\sigma(\A_4,\A_5)$ is not finite (see Remark \ref{sigmafinita}).

\begin{prop}\label{Bezoutlike}
Let $r_1,\dots, r_n \in\Z_{\geq 1}$. If $\sigma(\A_{r_1},\dots,
\A_{r_n})<\infty$ then $\sigma(\B_{r_1},\dots, \B_{r_n})<\infty$ and
$$
\sigma(\A_{r_1},\dots,\A_{r_n})=\sigma(\B_{r_1},\dots,
\B_{r_n})=\frac{r_1\cdots r_n}{w_1\cdots w_n}.
$$
\end{prop}

\begin{proof}
By Proposition \ref{sigmaexists}, there exists a sufficiently
general element $(h_1,\dots, h_n)\in \B_{r_1}\oplus \cdots \oplus
\B_{r_n}$ such that
\begin{equation}\label{Bh}
\sigma(\B_{r_1},\dots, \B_{r_n})=e(h_1,\dots, h_n).
\end{equation}

The condition $\sigma(\A_{r_1},\dots, \A_{r_n})<\infty$ implies that
$\A_{r_i}\neq 0$, for all $i=1,\dots, n$. The ideal $\A_{r_i}$ is
generated by the monomials of $w$-degree $r_i$, for all $i=1,\dots,
n$, then $h_i$ can be written as $h_i=g_i+g'_i$, for all
$i=1\dots,n$, where $(g_1,\dots, g_n)$ is a sufficiently general
element of $\A_{r_1}\oplus \cdots \oplus \A_{r_n}$ and $g'_i\in \O_n$
verifies that $d_w(g'_i)>r_i$, for all $i=1,\dots, n$. Therefore
$p_w(h_i)=g_i$, for all $i=1,\dots, n$.

Let $g$ denote the map $(g_1,\dots, g_n):(\C^n,0)\to (\C^n,0)$.  The
condition $\sigma(\A_{r_1},\dots, \A_{r_n})<\infty$ and the genericity
of $g$ imply that $g$ is finite, that is, $g^{-1}(0)=\{0\}$ and
$\sigma(\A_{r_1},\dots, \A_{r_n})=e(g_1,\dots, g_n)$.  Consequently
the map $h:(\C^n,0)\to (\C^n,0)$ is semi-weighted homogeneous with
respect to $w$.  By \cite[\S 12]{ArnoldGuseinVarchenko1985}
(see also \cite{BrianconMaynadier2002} for a more general phenomenon),
this implies that
$$
e(h_1,\dots, h_n)=e(g_1,\dots, g_n)=\frac{r_1\cdots r_n}{w_1\cdots
w_n}.
$$
Then the result follows.
\end{proof}

\begin{defn}\label{defwmatching}
Let $J_1,\dots, J_n$ be a family of ideals of $\O_n$ and let
$r_i=d_w(J_i)$, for all $i=1,\dots, n$.  We say that $J_1,\dots, J_n$
\textit{admits a $w$-matching} if there exists a permutation $\tau$
of $\{1,\ldots,n\}$ and an index $i_{0}\in\{1,\ldots,n\}$ such that
\begin{enumerate}
    \item[(a)]  $w_{i_{0}}=\min\{w_{1},\ldots,w_{n}\}$,

    \item[(b)]  $r_{\tau(i_{0})}=\max\{r_{1},\ldots,r_{n}\}$ and

    \item[(c)]  the pure monomial $x_{i}^{r_{\tau(i)}/w_i}$
    belongs to $J_{\tau(i)}$, for
    all $i\neq i_{0}$.
\end{enumerate}
\end{defn}

\begin{rem}\label{casosparticulares}
It is clear from the previous definition that if $r_1,\dots,
r_n\in\Z_{\geq 1}$ and $\A_{r_i}$ has finite colength, for all
$i=1,\dots, n$,
then $\A_{r_1},\dots, \A_{r_n}$ admits an $w$-matching.
If $r\in \Z_{\geq 1}$ then we observe that
$\A_r$ has finite colength if and only if $w_i$ divides $r$, for all
$i=1,\dots, n$.

Let us consider the case $n=2$ of the previous definition. Therefore,
let $r_1, r_2\in\Z_{\geq 1}$ such that $r_1>r_2$
and let us suppose that $w_1<w_2$. Let $J_1, J_2$ be ideals of $\O_2$
such that $d_w(J_i)=r_i$, $i=1,2$. Then $J_1,J_2$
admits an $w$-matching if and only if $y^{r_2/w_2}\in J_2$.

It is straightforward to observe that, under the conditions of the
previous definition, if $J_i$ contains the pure monomial
$x_i^{d_w(J_i)}$, for all $i=1,\dots, n$, then automatically
$J_1,\dots, J_n$ admits a $w$-matching.
\end{rem}

\begin{thm}\label{main}
Let $r_1,\dots, r_n \in\Z_{\geq 1}$ such that $\sigma(\A_{r_1},\dots,
\A_{r_n})<\infty$. Let $J_1,\dots, J_n$ be a
set of ideals of $\O_n$ such that $d_w(J_i)=r_i$, for all
$i=1,\dots, n$, and $\sigma(J_1,\dots, J_n)=\sigma(\A_{r_1},\dots,
\A_{r_n})$. Then
\begin{equation}\label{central}
\LL_0(J_1,\dots, J_n)\leq \LL_0(\A_{r_1},\dots,
\A_{r_n})\leq \LL_0(\B_{r_1},\dots, \B_{r_n}) \leq \frac{\max\{r_1,\dots,
r_n\}}{\min\{w_1,\dots, w_n\}}.
\end{equation}
and the above inequalities turn into equalities if $J_1,\dots, J_n$ admit an $w$-matching.
\end{thm}

\begin{proof}
The condition $\sigma(J_1,\dots, J_n)=\sigma(\A_{r_1},\dots,
\A_{r_n})$
and Proposition \ref{Bezoutlike} imply that
$$
\LL_0(J_1,\dots, J_n)\leq\LL_0(\A_{r_1},\dots, \A_{r_n})\leq
\LL_0(\B_{r_1},\dots, \B_{r_n}),
$$
by virtue of Proposition \ref{uppers}. Let us denote
$\max\{r_1,\dots, r_n\}$ and $\min\{w_1,\dots, w_n\}$ by $p$ and $q$,
respectively.
Let us see that $\LL_0(\B_{r_1},\dots, \B_{r_n})\leq\frac pq$.

Let us denote by $\overline w$ the product $w_1\cdots w_n$ and let us
consider the ideal $
J=\langle x_1^{\alpha_1},\dots,x_n^{\alpha_n}\rangle
$,
where $\alpha_i=\frac{\w}{w_i}$, for all $i=1,\dots, n$. Since
$\sigma(\B_{r_1},\dots, \B_{r_n})<\infty$, it makes sense to compute
the number $r_J(\B^s_{r_1},\dots, \B^s_{r_n})$, for all $s\geq 1$:
\begin{align*}
r_J(\B^s_{r_1},\dots, \B^s_{r_n})&=\min\left\{r\geq 1:
\sigma(\B^s_{r_1},\dots, \B^s_{r_n})=\sigma(\B^s_{r_1}+J^r,\dots,
\B^s_{r_n}+J^r)\right\}\\
&=\min\left\{r\geq 1: \frac{sr_1\cdots sr_n}{\w}=\frac{\min\{sr_1, \w
r\}\cdots \min\{sr_n, \w r\}}{\w}\right\}\\
&=\min\big\{r\geq 1: \w r\geq \max\{sr_1,\dots, sr_n\}\big\}\\
&=\min\left\{r\geq 1: r\geq \frac{\max\{sr_1,\dots,
sr_n\}}{\w}\right\}
=\left\lceil\frac{\max\{sr_1,\dots, sr_n\}}{\w}\right\rceil,
\end{align*}
where $\lceil a\rceil$ denotes the least integer greater than or
equal to $a$, for any $a\in \R$.
Therefore
\begin{align*}
\LL_J(\B_{r_1},\dots, \B_{r_n})&=
\inf_{s\geq 1}\frac{r_J(\B^s_{r_1},\dots, \B^s_{r_n})}{s}
\leq
\inf_{a\geq 1}\frac{r_J(\B^{a\w}_{r_1},\dots,\B^{a\w}_{r_n})}{a\w}\\
&=\inf_{a\geq 1}\frac{1}{a\w}
\left\lceil\frac{\max\{a\w r_1,\dots,a\w r_n\}}{\w}\right\rceil
=\frac{\max\{r_1,\dots, r_n\}}{\w}.
\end{align*}

Moreover, by Proposition \ref{intersecteixos} we have
$$
\LL_0(J)=\max\{\alpha_1,\dots, \alpha_n\}=\frac{\w}{\min\{w_1,\dots,
w_n\}},
$$
since $J$ is a monomial ideal.
Therefore, by Lemma \ref{transit} we obtain
\begin{align*}
\LL_0(\B_{r_1},\dots, \B_{r_n})&\leq
\LL_0(J)\LL_J(\B_{r_1},\dots,\B_{r_n}) \\
&\leq\frac{\w}{\min\{w_1,\dots, w_n\}}\frac{\max\{r_1,\dots,
r_n\}}{\w}
=\frac{\max\{r_1,\dots, r_n\}}{\min\{w_1,\dots, w_n\}}.
\end{align*}

Let us prove that $\LL_0(J_{1},\dots, J_{n})\geq \frac{p}{q}$ suposing
that $J_1,\dots, J_n$ admit an $w$-matching.  This inequality holds if
and only if
$$
\frac{r(J_1^s,\dots, J_n^s)}{s}\geq \frac pq,
$$
for all $s\geq 1$. By Lemma \ref{rpowers} we have that
$qr(J_1^s,\dots, J_n^s)\geq r(J_1^{sq},\dots, J_n^{sq})$, for all
$s\geq 1$. Therefore it suffices to show that
\begin{equation}\label{suficient1}
r(J_1^{sq},\dots, J_n^{sq})> sp-1,
\end{equation}
for all $s\geq 1$. Let us fix an integer $s\geq 1$, then relation
(\ref{suficient1}) is equivalent to saying that
\begin{equation}\label{suficient2}
\sigma(J_1^{sq},\dots, J_n^{sq})> e(J_1^{sq}+m^{sp-1}, \dots,
J_1^{sq}+m^{sp-1}).
\end{equation}

Since $J_1,\dots, J_n$ admits an $w$-matching,
let us consider a permutation $\tau$ of $\{1,\ldots,n\}$ such that
\begin{enumerate}
    \item[(a)]   $w_{i_{0}}=\min\{w_{1},\ldots,w_{n}\}$,

    \item[(b)]   $r_{\tau(i_{0})}=\max\{r_{1},\ldots,r_{n}\}$ and

    \item[(c)]   the pure monomial $x_{i}^{r_{\tau(i)}/w_i}$
    belongs to $J_{\tau(i)}$ for
    all $i\neq i_{0}$.
\end{enumerate}
Let us define the ideal
$$
H=\left\langle x_i^{\frac{r_{\tau(i)}sq}{w_i}}: i\neq i_0
\right\rangle +\left\langle x_{i_0}^{sp-1}\right\rangle.
$$

Let $i\in\{1,\dots,n\}$, $i\neq i_0$. Since
$x_i^{\frac{r_{\tau(i)}}{w_i}}\in J_{\tau(i)}$, for all $i\neq i_0$,
we have
\begin{equation}\label{laH}
e(H)\geq e(J_{\tau(1)}^{sq}+m^{sp-1},\dots,J_{\tau(n)}^{sq}+m^{sp-1})=
e(J_1^{sq}+m^{sp-1},\dots,J_n^{sq}+m^{sp-1}).
\end{equation}
Hence, if we prove that $\sigma(J^{sq}_1,\dots, J^{sq}_n)> e(H)$ then
the result follows.

We observe that
\begin{equation}\label{sigmaH}
\sigma(J^{sq}_1,\dots, J^{sq}_n)=(sq)^n\frac{r_1\cdots r_n}{w_1\cdots
w_n}\,, \hspace{1cm} e(H)=(sq)^{n-1}\frac{r_1\cdots
r_n}{r_{\tau(i_0)}}\frac{w_{i_0}}{w_1\cdots w_n}(sp-1).
\end{equation}

Thus, since we assume that $r_{\tau(i_0)}=p$ and $w_{i_0}=q$, we have that
$\sigma(J^{sq}_1,\dots, J^{sq}_n)>e(H)$ if and only if
$$
sq>\frac{q}{p}(sp-1),
$$
which is to say that $spq>spq-q$. Therefore relation
(\ref{suficient2}) holds for all integer $s\geq 1$ and consequently
the inequality $\LL_0(J_{r_1},\dots, J_{r_n})\geq \frac pq$ follows.
Thus relation (\ref{central}) is proven.
\end{proof}

\begin{rem}
We observe that the condition that $J_1,\dots, J_n$ admits an
$w$-matching
can not be removed from the hypothesis of the previous theorem.
Let us consider now the weighted homogeneous filtration in $\O_2$
induced by the vector of weights
$w=(1,4)$ and let $J_1,J_2$ be the ideals of $\O_2$ given by
$J_1=\langle x^4\rangle$, $J_2=\langle y^2\rangle$.
We observe that $d_w(x^4)=4$, $d_w(y^2)=8$ and consequently the right
hand side of (\ref{central}) would lead
to the conclusion that $\LL_0(J_1,J_2)=8$, which is not the case,
since clearly $\LL_0(x^4, y^2)=4$. We also observe
that the system of ideals $J_1, J_2$ does not admit an $w$-matching (see
Remark \ref{casosparticulares}).
\end{rem}

In order to simplify the exposition, we need to introduce the following definition.

\begin{defn}
If $f\in\O_n$, $f(0)=0$, then $f$ is termed {\it convenient} when
$\Gamma_+(f)$ intersects each coordinate axis.  Let $J_i$ denote the
ideal of $\O_n$ generated by all monomials $x^k$ such that
$k\in\Gamma_+(\partial f/\partial x_i)$, $i=1,\dots, n$.  Let us fix a
vector of weights $w\in\Z^n_{\geq 1}$.  Then we say that $f$ {\it
admits a $w$-matching} when the family of ideals $J_1,\dots, J_n$
admits a $w$-matching (see Definition \ref{defwmatching}).
\end{defn}

It is easy to observe that if a function $f\in\O_n$ is convenient and
quasi-homogeneous, then $f$ admits a $w$-matching (see Remark
\ref{casosparticulares}).

Let us fix a vector of weights $w=(w_1,\dots, w_n)\in\Z_{\geq 1}^n$
and an integer $d\geq 1$. Then we denote
by $\O(w;d)$ the set of all functions $f\in\O_n$ such that $f$ is
semi-weighted homogeneous with respect to $w$ of degree $d$.

\begin{rem}\label{semiwh}
From Definition \ref{defwmatching} we observe that a function
$f\in\O(w;d)$ admits a $w$-matching if and only if $p_w(f)$ admits a
$w$-matching.
\end{rem}

\begin{cor}\label{CorMain}
Let $f:(\C^n,0)\to (\C,0)$ be a semi-weighted homogeneous function of
degree $d$ with respect to the
weights $w_1,\dots, w_n$. Then
\begin{equation}\label{gradient}
\LL_0(f)\leq \frac{d-\min\{w_1,\dots,w_n\}}{\min\{w_1,\dots,w_n\}}
\end{equation}
and equality holds if $f$ admits an $w$-matching.
\end{cor}

\begin{proof}
Let $J_i$ denote the ideal of $\O_n$ generated by all monomials $x^k$
such that $k\in\Gamma_+(\partial f/\partial x_i)$, $i=1,\dots, n$.
Theorem \ref{base} shows that $\LL_0( f)=\LL_0(J_1,\dots, J_n)$.  We
observe that $d_w(J_i)=d-w_i$, for all $i=1,\dots, n$.  Then the
result arises as a direct application of Theorem \ref{main}.
\end{proof}

It has been proven recently by P{\l}oski et al.  \cite{KrasinskiOleksikPloski2009}
that equality (\ref{gradient}) holds for all weighted homogeneous function
$f:(\C^3,0)\to (\C,0)$ such that $f$ has an isolated singularity at
the origin, under the hypothesis that $2w_{i}\leq d$ for all $i$.

Given a vector of weights $w=(w_1,\dots, w_n)$ and a degree $d$, then
it is not always possible to find a weighted homogeneous function
$f:(\C^n,0)\to (\C,0)$ of degree $d$ with respect to $w$ such that
$f$ admits a $w$-matching, as the following example shows.

\begin{ex}
Let $w=(1,2,3)$ and $d=16$.  Let $f$ be a weighted homogeneous
function of degree $d$ with respect to $w$.  Let $J_i$ denote the
ideal of $\O_3$ generated by all monomials $x^k$ such that
$k\in\Gamma_+(\partial f/\partial x_i)$, for all $i=1,2,3$.  As a
direct consequence of Definition \ref{defwmatching}, if $J_1,J_2,J_3$
admits a $w$-matching, then $J_3$ contains a pure monomial of $x_2$ or
a pure monomial of $x_3$, which is impossible since $d_w(J_3)=13$ and
neither $2$ nor $3$ are divisors of $13$.

However we observe that $\O(w;d)\neq \emptyset$, since the function
$f(x_1,x_2,x_3)=x_{1}^{16}+x_{2}^{8}+x_{1}x_{3}^{5}$ belongs to $\O(w;d)$.
\end{ex}

\begin{prop} \label{wMatchEntero}
Let $d$, $w_{1},\ldots,w_{n}$ be non negative integers such
that $w_{i}$ divides $d$ for all $i=1,\ldots,n$.  Let
$f:(\C^n,0)\to (\C,0)$ be a weighted homogeneous function of degree
$d$ with respect to the weights $w_1,\dots, w_n$.  Let us assume that
$f$ has an isolated singularity at the origin.  Then there exists
a change of coordinates $\x$ in $(\C^n,0)$ of the form
$x_{i}=y_{i}+h_{i}(y_{1},\ldots,y_{n})$, where $h_i$ is a polynomial
in $y_1,\dots,y_n$, $i=1,\ldots,n$, such that:
\begin{enumerate}
\item the function $f\circ \x$ is convenient;

\item if $h_i\neq 0$, then the polynomial $h_{i}$ is weighted
homogeneous of degree $w_i$ with respect to $w$ and therefore
$f\circ\x$ is weighted homogeneous of degree $d$ with respect to $w$.
\end{enumerate}
\end{prop}

\begin{proof}
Since $f$ has an isolated singularity at the origin, for any
$i=1,\dots, n$ we can fix an index $k_{i}\in\{1,\dots, n\}$ such
that $x_{i}^{m_{i}}$ appears in the support of $\frac{\partial
f}{\partial x_{k_i}}$, where $m_{i}=\frac{d-w_{k_{i}}}{w_{i}}$,
which is to say that the monomial $x_{k_{i}}x_{i}^{m_{i}}$ appears
in the support of $f$.  Then $w_{i}$ divides $d-w_{k_{i}}$
and consequently $w_{i}$ divides $w_{k_{i}}$, since $w_i$
divides $d$ by assumption.

For all $j=1,\dots, n$, we set $L_j=\{i: k_i=j, i\neq j\}$. Let us define
\begin{equation}\label{aji}
    h_{j}=\begin{cases}
    \sum_{i\in L_j} a_{j,i} y_{i}^{w_j/w_i} &\textnormal{if $L_j\neq \emptyset$}\\
    0  &\textnormal{otherwise},
    \end{cases}
\end{equation}
where we suppose that ${\{a_{j,i}\}}_{j,i}$ is a generic choice of
coefficients in $\C$.  It is straightforward to see that, given an
index $j\in\{1,\dots, n\}$ such that $h_j\neq 0$, the polynomial
$h_{j}$ is weighted homogeneous of degree $w_{j}\).

Let us consider the map $\x:(\C^n,0)\to (\C^n,0)$, $\x(y_1,\dots,
y_n)=(x_1,\dots, x_n)$, given by
$$
 x_{j}=y_{j}+h_{j}(y), \hspace{0.5cm}\textnormal{for all $j=1,\ldots,n$}.
$$

We conclude that $\x$ is a local biholomorphism, the function $f \circ
\x$ is weighted homogeneous with respect to $w$ of degree $d$ and,
by the genericity of the coefficients $a_{j,i}$ in (\ref{aji}),
the pure monomial $y_i^{d/w_i}$ appears in the support of $f \circ \x\),
for all $i=1,\ldots,n$.  Hence the function $f \circ \x$ is
convenient.
\end{proof}

\begin{ex}
Set $w=(1,2,3,4,6)$ and $d=12$.  The polynomial
$f=x_{1}^{12}+x_{2}^{4}x_{4}+x_{4}^{3}+x_{3}^{2}x_{5}+x_{5}^{2}$ is
weighted homogeneous of degree 12.  Let $J_i$ denote the ideal of
$\O_5$ generated by all monomials $x^k$ such that
$k\in\Gamma_+(\partial f/\partial x_i)$, $i=1,\dots, 5$.  A
straightforward computation shows that
$$
J_1=\langle x_1^{11}\rangle,\hspace{0.3cm} J_2=\langle x_2^3x_4\rangle,
\hspace{0.3cm} J_3=\langle x_3x_5 \rangle, \hspace{0.3cm}
J_4=\langle x_2^4, x_4^2 \rangle,\hspace{0.3cm} J_5=\langle x_3^2, x_5 \rangle.
$$
Since the ideals $J_2$ and $J_{3}$ do not contain any pure monomial, the family
of ideals $\{J_i: i=1,\dots, 5\}$ does not admit a $w$-matching.

Following the proof of Proposition \ref{wMatchEntero}, we consider the
coordinate change $\x:(\C^5,0)\to (\C^5,0)$, given by: $x_{1}=y_{1}$,
$x_{2}=y_{2}$, $x_{3}=y_{3}$, $x_{4}=y_{4}+y_{2}^2$,
$x_{5}=y_{5}+y_{3}^2$.  Let $g=f\circ\x$ and let $J'_i$ denote the
ideal of $\O_5$ generated by all monomials $y^k$ such that
$k\in\Gamma_+(\partial g/\partial y_i)$, $i=1,\dots, 5$.  Then, as
shown in that proof, the function $g$ is convenient and therefore the
family of ideals $\{J_i': i=1,\dots, 5\}$ admits a $w$-matching.
\end{ex}

\begin{cor} \label{CorEntero}
Let $d$, $w_{1},\ldots,w_{n}$ be non negative integers such
that $w_{i}$ divides $d$ for all $i=1,\ldots,n$.
Let $f:(\C^n,0)\to (\C,0)$ be a semi-weighted homogeneous function of
degree $d$ with respect to the weights $w_1,\dots, w_n$. Then
\begin{equation*}
\LL_0(f)=\frac{d-\min\{w_1,\dots,w_n\}}{\min\{w_1,\dots,w_n\}}.
\end{equation*}
\end{cor}

\begin{proof}
Since $f$ is semi-weighted homogeneous, the principal part $p_w(f)$
has an isolated singularity at the origin.  Let $\x:(\C^n,0)\to
(\C^n,0)$ denote the analytic coordinate change obtained in
Proposition \ref{wMatchEntero} applied to $p_w(f)$.  The function
$p_w(f)\circ \x$ is weighted homogeneous of degree $d$ with respect to
$w$.  Therefore
$$
p_w(f)\circ \x=p_w(f\circ \x),
$$
which implies that $f\circ \x$ is a semi-weighted homogeneous
function.  Then, by Proposition \ref{wMatchEntero} and Remark
\ref{semiwh}, the function $f\circ \x$ admits a $w$-matching.  Thus we
obtain, by Corollary \ref{CorMain}, that
$$
\LL_0(f\circ \x)=\frac{d-\min\{w_1,\dots, w_n\}}{\min\{w_1,\dots w_n\}}.
$$
Then the result follows, since the local {\L}ojasiewicz exponent is a
bianalytic invariant.
\end{proof}

We remark that in Corollary \ref{CorEntero} we do not assume
$2w_{i}\leq d$ as in \cite{KrasinskiOleksikPloski2009}.
This assumption can not be eliminated from the main result of
\cite{KrasinskiOleksikPloski2009}, as the following example shows.

\begin{ex}
Let us consider the polynomial $f$ of $\O_3$ given by
$f=x_{1}x_{3}+x_{2}^2+x_{1}^{2}x_{2}$.  We observe that $f$ is
weighted homogeneous of degree $4$ with respect to the vector of
weights $w=(1,2,3)$.  The Jacobian ideal is $\langle
x_{1},x_{2},x_{3}\rangle$ so that $\LL_{0}(f)=1\neq 3$.  We
remark that it is easy to check that $f$ does not admit a
$w$-matching.
\end{ex}


\end{document}